\documentclass[11pt]{article}

\usepackage[margin=1in]{geometry}
\usepackage{amsmath,amssymb,amsthm}
\usepackage{mathtools}
\usepackage{booktabs}
\usepackage{tikz}
\usepackage[T1]{fontenc}
\usepackage{lmodern}
\usepackage{microtype}
\usepackage{hyperref}
\hypersetup{
  pdftitle={A Rank-Count Theory for the Combinatorial Discretizable Distance Geometry Problem},
  pdfauthor={Michael Souza, Wagner da Rocha, Carlile Lavor},
  colorlinks=true,
  linkcolor=blue,
  citecolor=blue,
  urlcolor=blue
}

\theoremstyle{plain}
\newtheorem{theorem}{Theorem}
\newtheorem{lemma}{Lemma}

\theoremstyle{definition}
\newtheorem{definition}{Definition}
\newtheorem{remark}{Remark}

\title{A Rank-Count Theory for the Combinatorial Discretizable Distance Geometry Problem}
\author{
Michael Souza\\
Universidade Federal do Cear\'a\\
\texttt{michael@ufc.br}
\and
Wagner da Rocha\\
Unicamp\\
\texttt{wdarocha@ime.unicamp.br}
\and
Carlile Lavor\\
Unicamp\\
\texttt{clavor@unicamp.br}
}
\date{}

\begin{document}

\maketitle

\begin{abstract}
The Distance Geometry Problem (DGP) asks for a geometric realization of a weighted graph with \(n\) vertices in \(\mathbb{R}^K\) such that Euclidean distances between vertices match the given edge weights. When a vertex order where every non-seed vertex has \(K\) predecessors inducing a clique is part of the input, the search space can be discretized via \(K\)-lateration, branching into a binary tree. In this subclass, known as the Combinatorial Discretizable Distance Geometry Problem (Combinatorial DDGP), the goal is to determine the number of realizations satisfying all distance constraints. While the number of realizations is almost always \(2^{n-K}\) when no additional distance constraints are present, a topological solution count in the presence of additional constraints has remained elusive. We develop an algebraic rank-count theory for the feasible binary branch codes, proving that under mirror-separated parameters they form an affine space over \(\mathbb{F}_2\) whenever a viable reference solution exists.
\end{abstract}

\section{Introduction}
\label{sec:intro}

The Distance Geometry Problem (DGP) asks for a geometric representation of a weighted graph in $\mathbb{R}^K$: given a simple undirected graph $G = (V, E)$ and an edge weight function $d: E \to \mathbb{R}_+$, find an embedding (or realization) $x: V \to \mathbb{R}^K$ such that $\|x_i - x_j\|_2 = d_{ij}$ for every $\{i,j\}\in E$~\cite{mucherino2012discretizable}. The problem appears in key applications such as protein structure determination from nuclear magnetic resonance (NMR) distance data~\cite{liberti2008branch}, sensor network localization~\cite{yemini1978positioning}, the geometry of nanostructures~\cite{mucherino2020analysis} and data science~\cite{liberti2020distance}.

The main mechanism that allows us to formulate the originally continuous Distance Geometry Problem (DGP) as a discrete one, called the Discretizable Distance Geometry Problem (DDGP), is the introduction of a vertex order on $V = \{1, \ldots, n\}$ that supports $K$-lateration. If each non-seed vertex $j > K$ has a predecessor set $U_j \subset \{1, \ldots, j-1\}$ of size $K$ with known pairwise distances, the position of $j$ is determined by the intersection of $K$ spheres centered at the coordinates of its predecessors~\cite{mucherino2012discretizable,liberti2014number}. Generically, this intersection contains at most two points. By sequentially applying this lateration step, the continuous search space is discretized into a binary search tree where valid embeddings correspond to paths. 

Depending on the properties of this vertex order, different discrete problem variants are defined. When every $U_j$ induces a clique in $G$, we obtain the Combinatorial DDGP \cite{abud2024impossible}; when the predecessors are further restricted to be contiguous, such that $U_j=\{j-K, \ldots, j-1\}$, we obtain the Discretizable Molecular Distance Geometry Problem (DMDGP).

These variants reveal a striking conflict between the tractability of order construction and the tractability of solution counting. On one extreme, while counting solutions is well-understood for the DMDGP~\cite{liberti2013counting}, finding a valid contiguous order is NP-complete~\cite{cassioli2015discretization}. On the other extreme, while general DDGP orders are easy to construct, the lack of predecessor cliques makes lateration weight-dependent, preventing topological solution counting~\cite{abud2024impossible}. The Combinatorial DDGP elegantly resolves this conflict: finding a valid order is fixed-parameter tractable for a fixed dimension $K$~\cite{cassioli2015discretization}, while the clique condition preserves reflection symmetries to keep the search space discrete and weight-independent. However, despite representing this ideal compromise, no method was previously known to determine the number of solutions for Combinatorial DDGP instances.

To address this challenge, this paper develops a complete algebraic rank-count theory for the feasible binary branch codes of ordered Combinatorial DDGP instances. We prove that, under mirror-separated parameters and assuming the existence of at least one viable reference solution, the feasible branch codes form an affine space over the finite field \(\mathbb{F}_2\). This algebraic structure allows us to determine the exact number of realizations using a rank formula, without having to explore the binary search tree.

Our contribution is threefold:
\begin{enumerate}
    \item We decompose the branch decisions into constrained and free bits, and represent constrained branch transformations using graph-derived cone generators and base generators.
    \item We introduce a labeled violation matrix that separates violations caused by different mirror cliques, keeping track of mirror-specific incompatibilities.
    \item We prove the \emph{Rank-Count Theorem}: relative to any reference solution $s^\ast$, the feasible constrained branch codes form an affine coset. Consequently, under these mirror-separation and nonemptiness conditions, the total number of realizations is given by a precise rank formula depending solely on the ordered template and the active-edge pattern.
\end{enumerate}

We organize the remainder of the paper as follows. Section~\ref{sec:related} reviews DMDGP symmetries, Combinatorial DDGP counts, DDGP recognition, and the counting barrier for general DDGP instances. Sections~\ref{sec:subproblems}--\ref{sec:violations} define branch codes, generators, and labeled violations. Section~\ref{sec:results} proves the main theorem and states the mirror-separation hypothesis at the point where it is used. Section~\ref{sec:example} gives a fully worked 7-vertex example.

\section{Related Work}
\label{sec:related}

We review three key areas of prior work that directly contextualize our contribution: partial-reflection symmetries in discretized distance geometry, complexity and counting barriers, and the contrast between algorithmic search and algebraic representations.

\subsection{Discretization and Partial-Reflection Symmetries}
The DGP has a continuous search space in general, but certain vertex orders turn it into a discrete search problem. In the DMDGP~\cite{lavor2012discretizable}, each vertex $j>K$ is adjacent to its $K$ immediate predecessors, so $U_j = \{j-K, \ldots, j-1\}$. These predecessor sets are contiguous, and the corresponding local cliques make it possible to place $j$ by $K$-lateration. We call the edges from each vertex to its prescribed predecessors the discretization edges $E_D$; all remaining distance constraints are pruning edges $E_P=E\setminus E_D$, used to test and discard incompatible branches.

Geometrically, each lateration step offers two candidate coordinates. These candidates are mirror images across the affine hyperplane spanned by the predecessor coordinates. Thus a branch choice in the binary search tree can be interpreted as a partial reflection: it fixes the vertices already placed and reflects the remaining suffix of the realization~\cite{mucherino2012exploiting}.

This viewpoint gives the DMDGP search tree a strong algebraic structure. For each admissible branching vertex $v$, one associates a generator $g_v$ that applies the corresponding partial reflection to the suffix $\{v,\ldots,n\}$ of a realization. Composing these generators produces a group action on the realizations explored by the binary search tree. Pruning edges restrict which suffix reflections remain valid symmetries, yielding a subgroup generated by the reflections that do not cross a pruning span~\cite{mucherino2012exploiting,liberti2013counting}.

This symmetry group explains the generic power-of-two solution counts in the exact-distance setting~\cite{liberti2011number,liberti2014number}. Because DMDGP reflections act on suffix intervals of the vertex order, solution counting and search can be described from graph topology rather than numerical edge weights~\cite{lavor2021optimality}. The same symmetries also support build-up, splitting, and search algorithms for DMDGP instances~\cite{mucherino2012exploiting,goncalves2021new}. 

\subsection{Complexity and Counting Barriers}
Relaxing the contiguity requirement of the vertex order leads to important computational trade-offs. Valid contiguous DMDGP orders are NP-complete to find for any fixed dimension $K \ge 2$~\cite{cassioli2015discretization}. In the general DDGP, the predecessor sets $U_j$ need not be consecutive; this broader order structure has been studied through the order-independent $K$-discretization graph and the order-dependent $K$-incident graph~\cite{abud2018k}. Valid DDGP orders, by contrast, can be found by fixed-parameter tractable algorithms with respect to $K$, and in polynomial time for fixed $K$~\cite{cassioli2015discretization}. However, this relaxation significantly complicates solution counting. 

First, Abud et~al.~\cite{abud2024impossible} proved that no purely combinatorial (weight-independent) counting method exists for the general DDGP: without predecessor cliques, the number of lateration outcomes can depend on the numerical edge weights. The Combinatorial DDGP subclass avoids this barrier by requiring predecessor sets to induce cliques, thereby recovering the positive weight-independent result where the count is almost always $2^{n-K}$ when no additional constraints are present~\cite{abud2024impossible}. Second, when additional constraints are present, counting solutions in the Combinatorial DDGP becomes highly non-trivial. Because predecessors are non-consecutive, branch reflections no longer act on suffix intervals; instead, a branch change at a vertex $q$ propagates dynamically through a directed successor dependency graph. This dynamic propagation is the key obstruction that has prevented topological counting in the presence of additional constraints, which we address in this paper.

\subsection{Algorithmic Search versus Algebraic Representations}
Traditional approaches to handling additional constraints (pruning edges) in discretized distance geometry are predominantly algorithmic. The Branch-and-Prune (BP) algorithm~\cite{liberti2008branch,mucherino2012discretizable} performs a depth-first traversal of the binary search tree, placing each new vertex by lateration and pruning a branch as soon as a distance constraint is violated. Symmetry-based variants reduce this traversal by using partial reflections to update candidate realizations across subproblems induced by pruning edges~\cite{goncalves2021new}. While these algorithms are highly effective and exhibit fixed-parameter tractability under certain branching-frequency conditions, they rely on numerical coordinate calculations and tree traversal. 

In contrast, symbolic and algebraic representations have been developed to study the underlying structure of the search space. Binary branch-code representations are standard for modeling paths in the binary search tree, and related binary formulations have been used to analyze symmetries, degrees of freedom, and coordinate-free distance constraints~\cite{mucherino2012exploiting,goncalves2021new,mucherino2020analysis}. Our work builds on this algebraic perspective. Instead of performing tree-based coordinate searches, we model the feasible branch decisions as an affine space over the finite field $\mathbb{F}_2$. By encoding branch transformations as graph-derived generators and recording their conflicts with additional constraints in a labeled violation matrix, we show that, once a feasible reference solution is available and mirror separation holds, the count of feasible branch codes is determined directly by a rank formula.

\section{Branch Codes and Pruning Constraints}
\label{sec:subproblems}

This section sets up the combinatorial objects used in the counting argument. We first fix the ordered predecessor system that produces the branch decisions, then use its dependency structure to identify which decisions can influence pruning constraints. This separates the branch-code space into constrained and free coordinates, and yields a smaller constrained code whose cardinality determines the total number of realizations.

Let the chosen order be built into the vertex labels, so that $V=\{1,\ldots,n\}$. For each vertex $i$, write $U_i$ for the predecessor set of $i$. Additionally, let $V_0=\{1,\ldots,K\}$ be the fixed initial clique in $\mathbb{R}^K$. For $i\le K$, set $U_i=\varnothing$. For every non-seed vertex $i > K$, the predecessor set $U_i \subseteq \{1,\ldots,i-1\}$ has size $K$ and induces a prescribed clique in $G$. Note that the predecessor relation defines a directed acyclic graph with an arc $j \to i$ whenever $j \in U_i$.

For the rest of the paper, the ordered predecessor family $U=(U_i)_{i\in V}$ is fixed. All predecessor closures are taken with respect to this family, so we suppress the subscript $U$ and write $\operatorname{Fix}$ and $\operatorname{Cone}$.

Two closures of this predecessor relation will be used throughout the construction. The first one looks backward: to determine the position of a vertex, we must also determine the predecessors needed to place it, their predecessors, and so on.

\begin{definition}[Fixing set]
\label{def:fixing}
The fixing set of a vertex $i$ is
\[
\operatorname{Fix}(i)
=
\{i\}\cup \bigcup_{u\in U_i}\operatorname{Fix}(u).
\]
For seed vertices this gives the base case $\operatorname{Fix}(i)=\{i\}$ because $U_i=\varnothing$.
\end{definition}

The forward closure captures propagation: changing a branch decision at a vertex can affect every later vertex that depends on it, directly or indirectly.

\begin{definition}[Dependency cone]
\label{def:cone}
The dependency cone of a vertex $i$ is
\[
\operatorname{Cone}(i)
=
\{i\}\cup \bigcup_{w:\,i\in U_w}\operatorname{Cone}(w),
\]
where $w$ ranges over vertices. If no later vertex has $i$ as a predecessor, this gives the terminal base case $\operatorname{Cone}(i)=\{i\}$.
\end{definition}

Let $E_D = \binom{V_0}{2} \cup \{\{u,i\} : i>K,\, u \in U_i\}$ denote the seed and discretization edges induced by the predecessor relation, and let $E_P = E\setminus E_D$ denote the pruning edges. The full branch-decision set is
\[
B = V\setminus V_0.
\]
Define the pruning-relevant vertex set
\[
L
=
\bigcup_{\{u,v\}\in E_P}
\left(\operatorname{Fix}(u)\cup \operatorname{Fix}(v)\right),
\]
and the pruning-constrained branch set
\[
B_{\mathrm{c}}=L\setminus V_0.
\]
The remaining branch decisions
\[
B_{\mathrm{f}}=B\setminus B_{\mathrm{c}},
\qquad
f=\lvert B_{\mathrm{f}}\rvert,
\]
are free with respect to pruning constraints: changing them preserves all discretization edges by construction and cannot change the endpoints of any pruning edge.

The active edge set is
\[
F = E_D[L] \cup E_P,
\qquad
E_D[L]=\{\{u,v\}\in E_D \mid u,v\in L\}.
\]
When $E_P=\varnothing$, we have $L=\varnothing$, $B_{\mathrm{c}}=\varnothing$, and all branch decisions are free.

Let $x(s)=(x_a(s))_{a\in V}$ denote the embedding determined by the full branch code $s\in\mathbb{F}_2^B$. Throughout, counts are branch-code counts relative to the fixed seed clique $V_0$; realizations are not additionally quotient-identified by global Euclidean congruences. The full solution code is
\[
\Xi
=
\{s \in \mathbb{F}_2^B \mid
\|x_a(s)-x_b(s)\|_2=d_{ab}
\text{ for every } \{a,b\}\in E\}.
\]
The constrained solution code is
\[
\Xi_F
=
\{s_{\mathrm{c}} \in \mathbb{F}_2^{B_{\mathrm{c}}} \mid
\|x_a(s)-x_b(s)\|_2=d_{ab}
\text{ for every } \{a,b\}\in F\},
\]
where $s$ is any full branch code whose restriction to $B_{\mathrm{c}}$ is $s_{\mathrm{c}}$. This is well-defined because branch decisions in $B_{\mathrm{f}}$ do not affect edges in $F$. Hence
\[
\lvert \Xi\rvert = 2^f \lvert \Xi_F\rvert.
\]
Our goal is to compute $\lvert \Xi\rvert$ under generic geometric assumptions without running tree-based branch search.

\section{Graph-Derived Generators}
\label{sec:generators}

This section defines the graph-derived generators that encode branch transformations and their associated mirror cliques. We introduce cone generators to represent the algebraic propagation of branch decisions within their downstream cones. Base generators are defined for each unique predecessor clique in $\mathcal{C}$. Finally, we construct the generator mask matrix to map linear combinations of these generators to their net branch masks.

A generator represents a branch transformation together with the mirror clique that explains its geometric action. Each generator is a pair
\[
g=(m_g, C_g),
\]
where $m_g \in \mathbb{F}_2^{B_{\mathrm{c}}}$ is a branch mask and $C_g$ is a predecessor clique. The mask tells which pruning-constrained branch bits are toggled; the clique tells which affine mirror is used for the reflection.

\begin{definition}[Cone generator]
\label{def:cone-generator}
For each decision vertex $q \in B_{\mathrm{c}}$, the cone generator is
\[
g_q = (m_q, U_q),
\qquad
\operatorname{supp}(m_q) = B_{\mathrm{c}} \cap \operatorname{Cone}(q).
\]
\end{definition}

When the rows and cone-generator columns are ordered by the DDGP vertex order, the cone-mask matrix is triangular with diagonal entries equal to one. Hence the cone masks $\{m_q \mid q\in B_{\mathrm{c}}\}$ form a basis of the constrained branch space $\mathbb{F}_2^{B_{\mathrm{c}}}$.

The predecessor-clique set used for base generators is
\[
\mathcal{C}=\{U_w\mid w\in B_{\mathrm{c}}\}.
\]

\begin{definition}[Base generator]
\label{def:base-generator}
For a predecessor clique $C\in\mathcal{C}$, define the generator roots:
\[
\operatorname{Gen}(C)=\{w \in B_{\mathrm{c}} \mid U_w = C\}.
\]
The base generator associated with $C$ is
\[
g_C=(m_C, C),
\]
where
\[
\operatorname{supp}(m_C)
=
B_{\mathrm{c}} \cap \bigcup_{w:\,U_w=C} \operatorname{Cone}(w)
=
\bigcup_{w\in \operatorname{Gen}(C)} \operatorname{supp}(m_w).
\]
\end{definition}

Base generators are highly valuable because a single algebraic branch mask can correspond to multiple generator presentations. We must retain these duplicate masks as distinct labeled columns of the mask and violation matrices rather than collapse them under a quotient operation. This is because the algebraic violation test for each generator is evaluated strictly relative to its mirror clique; the mask captures the coordinate branch-code shift, while the label records the specific affine mirror used by that algebraic certificate of feasibility.

After defining both kinds of generators, set $\mathcal{G}$ to be the labeled generator family
\[
\mathcal{G}
=
\{g_q\mid q\in B_{\mathrm{c}}\}\sqcup\{g_C\mid C\in\mathcal{C}\},
\qquad
m=\lvert\mathcal{G}\rvert.
\]
This is an indexed family of generator columns, not a quotient by branch mask. Duplicate masks are retained as separate columns when they carry distinct cone or base labels, and in particular when they carry distinct mirror labels. The generator coefficient space is
\[
A = \mathbb{F}_2^{\mathcal{G}}.
\]
For a mirror clique $C$, write
\[
\mathcal{G}_C=\{g\in\mathcal{G}: C_g=C\}.
\]
Choose an enumeration $\mathcal{G}=\{g^1,\ldots,g^m\}$ for the matrix representation below.

\begin{definition}[Mask matrix]
\label{def:mask-matrix}
The generator mask matrix is the linear map
\[
M: A \to \mathbb{F}_2^{B_{\mathrm{c}}},
\]
whose $j$-th column is $m_{g^j}$. For any generator coefficient vector $\alpha=(\alpha_1,\ldots,\alpha_m)\in A$, the action of $M$ is
\[
M\alpha = \bigoplus_{g^j:\,\alpha_j = 1} m_{g^j}.
\]
\end{definition}

\noindent Concretely, $M$ is the $\lvert B_{\mathrm{c}}\rvert\times m$ binary matrix whose columns are the generator masks:
\[
M \;=\;
\begin{bmatrix}
  \Big| & \Big| & & \Big| \\[4pt]
  m_{g^1} & m_{g^2} & \cdots & m_{g^m} \\[4pt]
  \Big| & \Big| & & \Big|
\end{bmatrix},
\]
where each column $m_{g^j}\in\mathbb{F}_2^{B_{\mathrm{c}}}$ is the branch mask of generator $g^j$.
Rows are indexed by the pruning-constrained branches $B_{\mathrm{c}}$, and columns are indexed by the generators $\mathcal{G}=\{g^1,\ldots,g^m\}$.

The vector $M\alpha$ represents the net branch mask produced by the generator combination $\alpha$.

\section{Labeled Violations}
\label{sec:violations}

This section formalizes how partial reflections (generator moves) can violate active distance constraints. 
Geometrically, a reflection only affects the length of an active edge if the edge crosses the boundary of the reflection support, meaning one endpoint is reflected while the other remains stationary. 
However, to build a consistent counting theory over $\mathbb{F}_2$, a violation cannot be modeled as a simple binary event on the edge. 
Instead, it must be indexed by both the edge and the mirror clique of the generator, yielding a \emph{labeled violation}. 
This dual indexing is algebraically indispensable: our algebraic certificate allows cancellations only within the same mirror label, while any potential cross-mirror coincidences are treated as exceptional and mathematically excluded by the mirror-separation hypothesis.

To formalize these geometric mechanisms algebraically, let us first establish how generator supports interact with the active constraints. Fix the active edge set $F$. Let $g=(m_g, C_g)$ be a generator, and write
\[
S_g = \operatorname{supp}(m_g).
\]
An active edge $e = \{a,b\} \in F$ crosses the support boundary of $g$ if exactly one endpoint is toggled:
\[
\lvert e \cap S_g \rvert = 1.
\]
If this happens, one endpoint is moved by the reflection while the other is fixed. For the generic parameter choices considered below, no active vertex outside a predecessor clique $C$ lies accidentally in $\operatorname{aff}(C)$. Under this generic non-incidence condition, the edge length is preserved exactly when the fixed endpoint belongs to $C_g$. Therefore, we say that $g$ violates $e$ with mirror label $C_g$ if
\[
\lvert e \cap S_g \rvert = 1
\quad\text{and}\quad
e\setminus S_g \not\subseteq C_g.
\]
Because $e$ crosses the support boundary, the set difference $e \setminus S_g$ is a singleton containing the unique endpoint of $e$ that remains fixed. Under generic non-incidence, the condition $e \setminus S_g \not\subseteq C_g$ is therefore equivalent to saying that this fixed endpoint does not lie on the reflection mirror hyperplane $\operatorname{aff}(C_g)$.

The labeled violations are indexed by both the edge and the mirror clique. Let
\[
\mathcal{L}
=
\{(e, C_g) \mid g \in \mathcal{G},\ e \in F,\ g \text{ violates } e\}.
\]
\begin{definition}[Labeled violation matrix]
\label{def:violation-matrix}
The labeled violation matrix is the linear map
\[
V: A \to \mathbb{F}_2^{\mathcal{L}},
\]
whose $j$-th column encodes which labeled pairs $(e,C)\in\mathcal{L}$ are violated by generator $g^j$.
Rows are indexed by labeled pairs $\ell_i=(e_i,C_i)\in\mathcal{L}$, and columns are indexed by generators $g^j$. The entries of $V$ are
\[
V_{ij} = 1
\quad\Longleftrightarrow\quad
C_i = C_{g^j} \quad\text{and}\quad g^j \text{ violates } e_i.
\]
\end{definition}

\noindent Explicitly, $V$ is the $\lvert\mathcal{L}\rvert\times m$ binary matrix whose columns are the labeled violation patterns of each generator:
\[
V \;=\;
\begin{bmatrix}
  \Big| & \Big| & & \Big| \\[4pt]
  v_{g^1} & v_{g^2} & \cdots & v_{g^m} \\[4pt]
  \Big| & \Big| & & \Big|
\end{bmatrix},
\]
where each column $v_{g^j}\in\mathbb{F}_2^{\mathcal{L}}$ has a $1$ in row $(e,C)$ if and only if $g^j$ violates active edge $e$ with mirror label $C=C_{g^j}$.
Rows are indexed by labeled pairs $\mathcal{L}=\{(e_1,C_1),\ldots\}$, and columns are indexed by the generators $\mathcal{G}=\{g^1,\ldots,g^m\}$.

Thus, $V\alpha$ represents the net labeled violation pattern of the generator combination $\alpha$. The same active edge may be crossed by transformations using different predecessor mirrors; such violations are kept in distinct rows and cannot cancel unless their mirror labels agree.

\section{Main Results}
\label{sec:results}

The central objective of this section is to establish a precise equivalence between the physical geometry of distance-preserving branch shifts and the algebraic structure of the binary subspace $\mathcal{K}_F \subseteq \mathbb{F}_2^{B_{\mathrm{c}}}$. Specifically, we show that the algebraic space $\mathcal{K}_F$ is constructively feasible—meaning every shift in $\mathcal{K}_F$ preserves all active edge lengths—and that, under a natural generic separation condition, no other distance-preserving shifts can exist.

Geometrically, moving from a branch code $s$ to $s \oplus h$ via a branch shift $h \in \mathbb{F}_2^{B_{\mathrm{c}}}$ corresponds to a sequence of partial reflections; whenever such a constrained shift is added to a full branch code, it is extended by zero on $B_{\mathrm{f}}$. A shift preserves the validity of all active constraints if it is produced by a generator combination with zero net labeled violations, meaning $\alpha \in \ker V$. Because multiple generator combinations can yield the same branch mask, the candidate feasible branch-shift space is the image of this kernel under the mask matrix $M$:
\[
\mathcal{K}_F = M(\ker V).
\]

\begin{lemma}[Branch-shift rank]
\label{lem:rank}
For the feasible branch shift space $\mathcal{K}_F = M(\ker V)$,
\[
\dim \mathcal{K}_F
=
\operatorname{rank}_{\mathbb{F}_2}\begin{bmatrix} M \\ V \end{bmatrix}
-
\operatorname{rank}_{\mathbb{F}_2}(V).
\]
\end{lemma}

\begin{proof}
Let $m = \lvert \mathcal{G} \rvert = \dim A$. Since $\mathcal{K}_F = M(\ker V)$, we restrict $M$ to $\ker V$:
\[
M\big|_{\ker V} : \ker V \to \mathbb{F}_2^{B_{\mathrm{c}}}.
\]
By rank-nullity applied to this restricted map,
\[
\dim \mathcal{K}_F
=
\dim \ker V
-
\dim(\ker V \cap \ker M).
\]
Now $\dim \ker V = m - \operatorname{rank}_{\mathbb{F}_2}(V)$, and
\[
\ker V \cap \ker M
=
\ker\begin{bmatrix} M \\ V \end{bmatrix}.
\]
Therefore,
\[
\dim(\ker V \cap \ker M)
=
m - \operatorname{rank}_{\mathbb{F}_2}\begin{bmatrix} M \\ V \end{bmatrix}.
\]
Substituting these dimensions into the rank-nullity equation gives:
\[
\begin{aligned}
\dim \mathcal{K}_F &= (m - \operatorname{rank}_{\mathbb{F}_2}(V)) - \left(m - \operatorname{rank}_{\mathbb{F}_2}\begin{bmatrix} M \\ V \end{bmatrix}\right) \\
&= \operatorname{rank}_{\mathbb{F}_2}\begin{bmatrix} M \\ V \end{bmatrix} - \operatorname{rank}_{\mathbb{F}_2}(V).
\end{aligned}
\]
This completes the proof.
\end{proof}

The next lemma isolates the local geometric step needed below, showing that algebraic branch flips correspond to valid Euclidean reflections under suitable compatibility conditions.

\begin{definition}[Mirror-compatibility]
\label{def:mirror-compatible}
Let $C$ be a predecessor clique, $S \subseteq B_{\mathrm{c}}$ be a subset of constrained branches, and $F$ be the active edge set. The block $(S,C)$ is \emph{mirror-compatible} with $F$ if, for every active edge $e = \{a,b\} \in F$ that crosses the block boundary (i.e., $\lvert e \cap S \rvert = 1$), the fixed endpoint belongs to $C$, meaning $e \setminus S \subseteq C$.
\end{definition}

Mirror-compatibility is a purely topological and combinatorial requirement on how the block boundary interacts with the active constraints. The next lemma shows that when a block satisfies this combinatorial condition alongside appropriate causality constraints, the algebraic branch operation is geometrically realized by a valid Euclidean reflection that preserves all active distance constraints.

\begin{lemma}[Block admissibility]
\label{lem:block-admissibility}
Let $C$ be a predecessor clique with $r(C)=\min\{q\in B_{\mathrm{c}}: U_q=C\}$, and let $S\subseteq B_{\mathrm{c}}$ satisfy $S\subseteq \{r(C),\ldots,n\}$ and $S\cap C=\varnothing.$ If $(S,C)$ is mirror-compatible with the active edge set $F$ (Definition~\ref{def:mirror-compatible}), then for any nondegenerate branch embedding $x(s)$ (with $s_{\mathrm{c}}=s|_{B_{\mathrm{c}}}$), let $R_C$ be the Euclidean reflection across the affine hull $\operatorname{aff}\{x_c(s) : c \in C\}$. Then, for $i \in L$, the coordinates
\[
y_i =
\begin{cases}
R_C(x_i(s)), & i\in S,\\
x_i(s), & i\in L\setminus S
\end{cases}
\]
are obtained on the active subsystem by the constrained branch code $s_{\mathrm{c}}\oplus \mathbf{1}_S$, and every active edge length is preserved:
\[
\|y_a-y_b\|_2=\|x_a(s)-x_b(s)\|_2
\qquad
\text{for all } \{a,b\}\in F.
\]
\end{lemma}

\begin{proof}
First consider the seed and discretization edges in the active subsystem; seed-clique edges are fixed. If $i\in S$, then $i\in B_{\mathrm{c}}$, hence $i\in L$. Since $L$ is closed under fixing sets, every predecessor $u\in U_i$ also lies in $L$, so the discretization edge $\{u,i\}$ belongs to $E_D[L]\subseteq F$. If $u\notin S$, then $\{u,i\}$ crosses the block boundary, and mirror-compatibility gives $u\in C$. Thus we have $U_i\subseteq S\cup C$. Consequently, the transformed predecessor simplex of $i$ is the image of the original predecessor simplex under the same reflection $R_C$: predecessors in $S$ are reflected, and predecessors in $C$ are fixed by $R_C$.

If $i\in L\setminus (S\cup C)$, then no predecessor of $i$ can lie in $S$. Indeed, if some $u\in U_i\cap S$, then the active edge $\{u,i\}\in E_D[L]$ would cross the block boundary, and mirror-compatibility would force $i\in C$, a contradiction. Hence, the predecessor simplex of such an active vertex is unchanged.

It follows, by induction in the DDGP order restricted to $L$, that the coordinates $y_i$ are produced by valid lateration choices on the active subsystem. Vertices in $L\setminus S$ keep the same local lateration choice. Vertices in $S$ have both their coordinate and their predecessor simplex transformed by the reflection $R_C$; since $R_C$ reverses the local orientation across the predecessor affine hull, the corresponding lateration branch bit is toggled. Therefore the resulting constrained branch code is $s_{\mathrm{c}}\oplus \mathbf{1}_S$.

It remains to check active edge lengths. Let $e=\{a,b\}\in F$. If both endpoints belong to $S$, then both are acted on by the same Euclidean isometry $R_C$, so the length of $e$ is preserved. If neither endpoint belongs to $S$, both endpoints are fixed, so the length is preserved. Finally, suppose exactly one endpoint belongs to $S$, say $a\in S$ and $b\notin S$. By mirror-compatibility, $b\in C$, hence $R_C$ fixes $x_b(s)$. Therefore
\[
\|y_a-y_b\|_2
=
\|R_C(x_a(s))-x_b(s)\|_2
=
\|R_C(x_a(s))-R_C(x_b(s))\|_2
=
\|x_a(s)-x_b(s)\|_2.
\]
Thus every edge in $F$ is preserved.
\end{proof}

The following lemma establishes the crucial link between the algebraic kernel of the violation matrix and the geometric compatibility of the corresponding deformed blocks. Specifically, any selection of generators in the kernel of $V$ yields a set of blocks that are combinatorially mirror-compatible with the active constraints.

\begin{lemma}[Kernel-induced mirror-compatibility]
\label{lem:kernel-compatibility}
Let $\alpha \in \ker V$. For any mirror clique $C$, let $\alpha_C = \{g \in \mathcal{G}_C : \alpha_g = 1\}$ be the set of selected generators with mirror $C$, and let $S_C = \bigoplus_{g \in \alpha_C} S_g$ be the symmetric difference of their supports. Then the block $(S_C, C)$ is mirror-compatible with the active edge set $F$ in the sense of Definition~\ref{def:mirror-compatible}.
\end{lemma}

\begin{proof}
Suppose that an active edge $e = \{a,b\} \in F$ crosses $S_C$, meaning that exactly one endpoint lies in $S_C$. Without loss of generality, assume $a \in S_C$ and $b \notin S_C$. Let $\chi_g: V \to \{0,1\}$ denote the indicator function of $S_g$ for each $g \in \alpha_C$. The indicator function of $S_C$ satisfies
\[
\chi_{S_C}(v) \equiv \sum_{g \in \alpha_C} \chi_g(v) \pmod 2.
\]
Because generator supports are disjoint from their own mirror cliques, we have $S_g \cap C = \varnothing$ for all $g \in \alpha_C$, which implies $S_C \cap C = \varnothing$ and hence $a \notin C$.

Now, suppose for contradiction that the fixed endpoint $b$ is also not in $C$. A generator $g \in \alpha_C$ violates $e$ with label $C$ if and only if it crosses $e$ (i.e., $\lvert e \cap S_g \rvert = 1$) and the fixed endpoint $e \setminus S_g$ lies outside $C$. Since $a \notin C$ and $b \notin C$, the latter condition is satisfied for every crossing generator. The number of such crossing generators in $\alpha_C$ is
\[
\sum_{g \in \alpha_C} \lvert e \cap S_g \rvert
=
\sum_{g \in \alpha_C} (\chi_g(a) + \chi_g(b))
\equiv
\chi_{S_C}(a) + \chi_{S_C}(b)
\equiv
1 + 0
\equiv
1 \pmod 2.
\]
Thus, an odd number of generators in $\alpha_C$ violate $e$ with label $C$, contributing an odd parity to the row $(e,C)$ of the violation product $V\alpha$. This implies $(V\alpha)_{(e,C)} \equiv 1 \pmod 2$, which directly contradicts the assumption that $\alpha \in \ker V$ (which requires $V\alpha = 0$). Hence, the fixed endpoint $b$ must lie in $C$, establishing that the block $(S_C, C)$ is mirror-compatible with $F$.
\end{proof}

Using this algebraic-to-geometric bridge, we can now prove that any branch shift in the range of the kernel of $V$ preserves all active edge lengths, meaning it maps valid configurations to valid configurations.

\begin{lemma}[Zero-violation preservation]
\label{lem:preservation}
Let $\alpha \in \ker V$ and let $h = M\alpha$. If $s_{\mathrm{c}} \in \Xi_F$, then $s_{\mathrm{c}} \oplus h \in \Xi_F$.
\end{lemma}

\begin{proof}
Fix a full extension $s\in\mathbb{F}_2^B$ of $s_{\mathrm{c}}$; active-edge lengths are independent of free bits. We group the selected generators in $\alpha$ by mirror clique. For each mirror clique $C$ that appears among the generators, let $\alpha_C = \{g \in \mathcal{G}_C : \alpha_g = 1\}$ and let $S_C = \bigoplus_{g \in \alpha_C} S_g$ be the symmetric difference of their supports. By Lemma~\ref{lem:kernel-compatibility}, the block $(S_C, C)$ is mirror-compatible with the active edge set $F$. Moreover, since $S_g \cap C = \varnothing$ for all $g \in \alpha_C$, we have $S_C \cap C = \varnothing$.

For any mirror clique $C$ that appears among the generators, let
\[
r(C)=\min\{q\in B_{\mathrm{c}}: U_q=C\}
\]
be the first constrained vertex generated from $C$. Order the nonzero mirror cliques $C_1,\ldots,C_k$ so that $r(C_1)<\cdots<r(C_k)$. Because the support of any generator with mirror $C_i$ only contains vertices discretized at or after $r(C_i)$, we have $S_{C_i}\subseteq\{r(C_i),\ldots,n\}$, so Lemma~\ref{lem:block-admissibility} applies to each block.

Apply the corresponding admissible blocks in reverse order $C_k,\ldots,C_1$. If $j>i$, then $S_{C_j}\subseteq\{r(C_j),\ldots,n\}$ while $C_i\subseteq\{1,\ldots,r(C_i)-1\}$, so $S_{C_j}\cap C_i=\varnothing$. Thus later blocks do not move the mirror clique needed by earlier blocks. Each application preserves all active edge lengths by Lemma~\ref{lem:block-admissibility}. The accumulated constrained branch mask is
\[
\bigoplus_C \mathbf{1}_{S_C}
=
\bigoplus_{g:\,\alpha_g=1} m_g
=
\bigoplus_{g^j:\,\alpha_j=1} m_{g^j}
=
M\alpha
=
h.
\]
Therefore the final constrained code is $s_{\mathrm{c}}\oplus h$, and since all active edge lengths have been preserved, $s_{\mathrm{c}}\oplus h\in\Xi_F$.
\end{proof}

Lemma~\ref{lem:preservation} proves the constructive inclusion: shifts in $\mathcal{K}_F$ generate feasible constrained codes from any feasible reference code. The converse could fail only if a shift outside $\mathcal{K}_F$ preserved all active distances by an accidental geometric coincidence. The following hypothesis excludes exactly those extra solutions.

\begin{definition}[Nondegenerate DDGP parameters]
\label{def:nondegenerate-params}
Fix an ordered DDGP template, which determines $M$, $V$, and $\mathcal{K}_F = M(\ker V)$. A parameter choice $\theta$ for this template is nondegenerate if every lateration step is well defined and, whenever the predecessors of a non-seed vertex have been placed, the corresponding intersection of $K$ spheres consists of two distinct reflected choices. 
\end{definition}

Let $p$ be the dimension of the chosen parameter space for the fixed template. We denote by $\Theta \subseteq \mathbb{R}^p$ the set of nondegenerate parameter choices; in the usual Euclidean parameterization this set is open and dense. For any parameter vector $\theta \in \Theta$, we write $x^\theta(s)$ for the unique physical lateration embedding associated with the branch code $s \in \mathbb{F}_2^B$.

While nondegeneracy ensures that the binary search tree of realizations is topologically well-defined at every step, it is a local condition. It does not prevent "accidental" global distance preservation where two distinct branches, separated by a reflection shift $h \notin \mathcal{K}_F$ not justified by the symmetry group, happen to satisfy all pruning constraints due to an extraordinary numerical alignment of the parameters. To mathematically rule out these exceptional coordinate alignments and guarantee that the algebraic kernel $\mathcal{K}_F$ characterizes \emph{exactly} the viable branch moves, we introduce a stronger, global geometric requirement.

\begin{definition}[Mirror-separated parameters]
\label{def:mirror-sep}
A nondegenerate parameter choice $\theta \in \Theta$ is mirror-separated if, for every constrained shift $h\in\mathbb{F}_2^{B_{\mathrm{c}}}\setminus\mathcal{K}_F$ (extended by zero on $B_{\mathrm{f}}$) and every full branch code $s\in\mathbb{F}_2^B$, there exists an active edge $\{a,b\} \in F$ such that
\[
\|x^\theta_a(s \oplus h) - x^\theta_b(s \oplus h)\|_2^2 \neq \|x^\theta_a(s) - x^\theta_b(s)\|_2^2.
\]
\end{definition}

\begin{remark}[A genericity criterion for mirror separation]
\label{rem:mirror-separation-genericity}
Mirror separation rules out accidental geometric coincidences in which a shift $h\notin\mathcal{K}_F$ nevertheless preserves all active edge distances. Fix the ordered template, the active subsystem $L$, and the active edge set $F$. For each constrained branch shift $h\in\mathbb{F}_2^{B_{\mathrm{c}}}\setminus\mathcal{K}_F$ and active edge $e=\{a,b\}\in F$, consider the discrepancy function
\[
\Delta_{h,e}(\theta;s)
=
\|x^\theta_a(s\oplus h)-x^\theta_b(s\oplus h)\|_2^2
-
\|x^\theta_a(s)-x^\theta_b(s)\|_2^2,
\]
defined on the nondegenerate parameter choices $\theta\in\Theta$ for which the corresponding branch embeddings exist. Mirror separation fails exactly when, for some shift $h\notin\mathcal{K}_F$ and some branch code $s$, all active-edge discrepancies vanish:
\[
\Delta_{h,e}(\theta;s)=0
\qquad
\text{for every } e\in F.
\]
Consequently, if for every $h\notin\mathcal{K}_F$ and every branch code $s$ at least one discrepancy $\Delta_{h,e}(\cdot;s)$ is not identically zero on the relevant nondegenerate parameter component, then the failure set for mirror separation is contained in a finite union of proper exceptional sets defined by the vanishing of nonzero discrepancy functions. Under this non-identity criterion, mirror separation holds for almost all such nondegenerate parameters, in the same Lebesgue-measure sense used for generic DMDGP statements and zero-measure pruning coincidences in~\cite{liberti2012discretizable}.
\end{remark}

Finally, the main theorem states that, under mirror separation, the converse also holds on the pruning-constrained branch set. The full count then follows by multiplying by the free branch decisions in $B_{\mathrm{f}}$.

\begin{theorem}[Rank-count theorem]
\label{thm:main}
Assume $K \ge 2$ and $\Xi \ne \varnothing$. Assume also that the present predecessor and edge-length parameters are mirror-separated in the sense of Definition~\ref{def:mirror-sep}. Then, for any reference solution $s^\ast \in \Xi$, with constrained restriction $s^\ast_{\mathrm{c}}=s^\ast|_{B_{\mathrm{c}}}$,
\[
\Xi_F
=
s^\ast_{\mathrm{c}} \oplus \mathcal{K}_F
=
s^\ast_{\mathrm{c}} \oplus M(\ker V).
\]
Consequently,
\[
\lvert \Xi\rvert
=
2^{
f+
\operatorname{rank}_{\mathbb{F}_2}\begin{bmatrix} M \\ V \end{bmatrix}
-
\operatorname{rank}_{\mathbb{F}_2}(V)
}.
\]
\end{theorem}

\begin{proof}
Fix $s^\ast \in \Xi$ and let $s^\ast_{\mathrm{c}} = s^\ast|_{B_{\mathrm{c}}}$.

$(\supseteq)$ For any $h \in \mathcal{K}_F$, Lemma~\ref{lem:preservation} gives
$s^\ast_{\mathrm{c}} \oplus h \in \Xi_F$, so $s^\ast_{\mathrm{c}} \oplus \mathcal{K}_F \subseteq \Xi_F$.

$(\subseteq)$ Let $s_{\mathrm{c}} \in \Xi_F$ and set $h = s_{\mathrm{c}} \oplus s^\ast_{\mathrm{c}}$.
Extend $h$ by zero on $B_{\mathrm{f}}$ and put $\widetilde{s}=s^\ast\oplus h$. Then $\widetilde{s}$ has constrained part $s_{\mathrm{c}}$ and is active-feasible. If $h \notin \mathcal{K}_F$, mirror separation gives an active-edge length difference between $s^\ast$ and $\widetilde{s}$, a contradiction. Hence $h \in \mathcal{K}_F$ and
$\Xi_F \subseteq s^\ast_{\mathrm{c}} \oplus \mathcal{K}_F$.

Together, $\Xi_F = s^\ast_{\mathrm{c}} \oplus M(\ker V)$. Since free branches on
$B_{\mathrm{f}}$ combine independently, $\lvert\Xi\rvert = 2^f \lvert\Xi_F\rvert
= 2^{f + \dim\mathcal{K}_F}$, and Lemma~\ref{lem:rank} gives the rank formula.
\end{proof}

\section{A Small Worked Example}
\label{sec:example}

This section details a fully worked seven-vertex example to concretely illustrate our theoretical framework. While small in scale, this example exhibits all the core features of the general theory: both cone and base generators are present, the labeled violation kernel $\ker V$ is nontrivial, a presentation redundancy in the generator combinations is factored out by mapping the kernel through the mask projection to obtain the feasible branch shift space $\mathcal{K}_F = M(\ker V)$, and one free branch decision contributes a nontrivial factor $2^f$ to the final count.

\subsection{Combinatorial Setup}
\label{sec:example-combo}

Let $K=2$ in $\mathbb{R}^2$. For readability in this example, write $v_i$ for the vertex with global label $i$. The fixed initial clique is $V_0=\{v_1, v_2\}$. Consider the predecessor sets
\[
U_3=\{v_1, v_2\},
\qquad
U_4=\{v_1, v_3\},
\qquad
U_5=\{v_1, v_3\},
\qquad
U_6=\{v_1, v_4\},
\qquad
U_7=\{v_1, v_2\},
\]
and the pruning edge set
\[
E_P=\{\{v_3, v_6\},\,\{v_4, v_5\}\}.
\]
The pruning-relevant vertex set is
\[
L=\{v_1,v_2,v_3,v_4,v_5,v_6\}.
\]
The added vertex $v_7$ is not in the fixing set of any pruning endpoint, so the constrained and free branch sets are
\[
B_{\mathrm{c}}=\{v_3, v_4, v_5, v_6\},
\qquad
B_{\mathrm{f}}=\{v_7\},
\qquad
f=1.
\]
There are three predecessor cliques,
\[
C_0=\{v_1, v_2\},
\qquad
C_1=\{v_1, v_3\},
\qquad
C_2=\{v_1, v_4\}.
\]
The first is the predecessor clique for $v_3$ and the free vertex $v_7$, but only $v_3$ contributes to the constrained generator masks. The second generates $v_4$ and $v_5$, and the third generates $v_6$. Since $v_4$ and $v_5$ depend on $v_3$, and $v_6$ depends on $v_4$, we have
\[
\operatorname{Cone}(v_3)\cap B_{\mathrm{c}}=\{v_3,v_4,v_5,v_6\},
\qquad
\operatorname{Cone}(v_4)\cap B_{\mathrm{c}}=\{v_4, v_6\}.
\]
The free vertex has $\operatorname{Cone}(v_7)\cap B_{\mathrm{c}}=\varnothing$, so its branch decision does not appear in the constrained matrices below.
The graph structure, predecessor dependency arcs, and pruning edges are illustrated in Figure~\ref{fig:worked-example}.

\begin{figure}[htbp]
\centering
\begin{tikzpicture}[
    scale=1.05,
    >=stealth,
    every node/.style={circle, draw, minimum size=8mm, inner sep=0pt, font=\small},
    seed/.style={fill=purple!15, draw=purple!70!black, line width=0.8pt},
    constrained/.style={fill=green!12, draw=green!55!black, line width=0.8pt},
    free/.style={fill=yellow!20, draw=orange!75!black, line width=0.8pt},
    dep/.style={->, thick, black!70},
    prune/.style={dashed, ultra thick, red!85!black}
]
    % Nodes
    \node[seed] (v1) at (0,0) {$v_1$};
    \node[seed] (v2) at (0,-2.25) {$v_2$};
    \node[constrained] (v3) at (4.05,0) {$v_3$};
    \node[constrained] (v4) at (0,2.25) {$v_4$};
    \node[constrained] (v5) at (4.05,2.25) {$v_5$};
    \node[constrained] (v6) at (4.05,-2.25) {$v_6$};
    \node[free] (v7) at (-2.25,-2.25) {$v_7$};
    
    % Initial clique edge
    \draw[thick] (v1) -- (v2);
    
    % Predecessor arcs (directed)
    \draw[dep] (v1) to[bend left=10] (v3);
    \draw[dep] (v2) to[bend right=28] (v3);

    \draw[dep] (v1) to (v4);
    \draw[dep] (v3) to (v4);
    
    \draw[dep] (v1) to[bend left=10] (v5);
    \draw[dep] (v3) to (v5);
    
    \draw[dep] (v1) to[bend right=4] (v6);
    \draw[dep] (v4) to (v6);

    \draw[dep] (v1) to (v7);
    \draw[dep] (v2) to (v7);
    
    % Pruning edges (dashed red)
    \draw[prune] (v4) to (v5);
    \draw[prune] (v3) to (v6);
\end{tikzpicture}
\caption{Combinatorial predecessor dependency graph and pruning edges for the 7-vertex worked example; node positions are schematic. Directed arcs represent predecessor discretization constraints ($E_D$), dashed red lines represent $E_P = \{\{v_3, v_6\}, \{v_4, v_5\}\}$, purple nodes form $V_0$, green nodes are branch-decision vertices in $B_{\mathrm{c}}$, and the yellow node is the free branch vertex $v_7\in B_{\mathrm{f}}$.}
\label{fig:worked-example}
\end{figure}
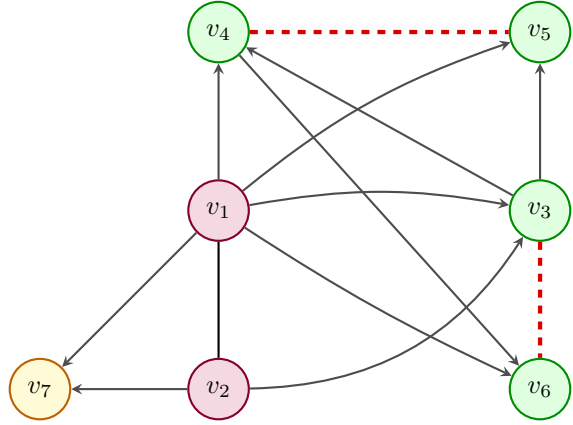

The generator supports and associated mirror cliques are listed in Table~\ref{tab:generators}. The constrained rows are ordered as $(v_3,v_4,v_5,v_6)$, and the columns are ordered as $(g_3,g_4,g_5,g_6,g_{C_0},g_{C_1},g_{C_2})$. The free branch bit $v_7$ is intentionally omitted from these constrained masks.

\begin{table}[htbp]
\centering
\caption{Generator supports and mirror cliques for the worked example. Each support mask $S_g$ is a subset of $B_{\mathrm{c}} = \{v_3, v_4, v_5, v_6\}$; the rows define the columns $(g_3,g_4,g_5,g_6,g_{C_0},g_{C_1},g_{C_2})$ of the mask matrix $M$.}
\label{tab:generators}
\begin{tabular}{lcc}
\toprule
Generator & Support Mask ($S_g$) & Mirror Clique ($C_g$) \\
\midrule
$g_3$ & $\{v_3, v_4, v_5, v_6\}$ & $C_0 = \{v_1, v_2\}$ \\
$g_4$ & $\{v_4, v_6\}$ & $C_1 = \{v_1, v_3\}$ \\
$g_5$ & $\{v_5\}$ & $C_1 = \{v_1, v_3\}$ \\
$g_6$ & $\{v_6\}$ & $C_2 = \{v_1, v_4\}$ \\
$g_{C_0}$ & $\{v_3, v_4, v_5, v_6\}$ & $C_0 = \{v_1, v_2\}$ \\
$g_{C_1}$ & $\{v_4, v_5, v_6\}$ & $C_1 = \{v_1, v_3\}$ \\
$g_{C_2}$ & $\{v_6\}$ & $C_2 = \{v_1, v_4\}$ \\
\bottomrule
\end{tabular}
\end{table}

The generator mask matrix is
\[
M=
\begin{bmatrix}
1&0&0&0&1&0&0\\
1&1&0&0&1&1&0\\
1&0&1&0&1&1&0\\
1&1&0&1&1&1&1
\end{bmatrix}.
\]

The active edge set for the constrained subproblem is $F = E_D[L] \cup E_P$. Besides the fixed seed clique edge $\{v_1,v_2\}$, the discretization edges induced on $L$ are
\[
\{v_1, v_3\},\{v_2, v_3\},
\{v_1, v_4\},\{v_3, v_4\},
\{v_1, v_5\},\{v_3, v_5\},
\{v_1, v_6\},\{v_4, v_6\}.
\]
The nonzero labeled active-edge violations are summarized in Table~\ref{tab:violations}.

\begin{table}[htbp]
\centering
\caption{Labeled active-edge violations for each generator. The entry \textnormal{none} means that the generator contributes no row to the labeled violation matrix $V$; the nonzero rows are indexed by $(\{v_4,v_5\},C_1)$ and $(\{v_3,v_6\},C_2)$.}
\label{tab:violations}
\begin{tabular}{lcc}
\toprule
Generator & Violated Active Edge ($e$) & Mirror Label ($C_g$) \\
\midrule
$g_3$ & \text{none} & \text{none} \\
$g_4$ & $\{v_4, v_5\}$ & $C_1$ \\
$g_5$ & $\{v_4, v_5\}$ & $C_1$ \\
$g_6$ & $\{v_3, v_6\}$ & $C_2$ \\
$g_{C_0}$ & \text{none} & \text{none} \\
$g_{C_1}$ & \text{none} & \text{none} \\
$g_{C_2}$ & $\{v_3, v_6\}$ & $C_2$ \\
\bottomrule
\end{tabular}
\end{table}

Thus, with rows ordered as the two active-edge labels $(\{v_4, v_5\},C_1)$ and $(\{v_3, v_6\},C_2)$, the labeled violation matrix is
\[
V=
\begin{bmatrix}
0&1&1&0&0&0&0\\
0&0&0&1&0&0&1
\end{bmatrix}.
\]
Write a generator coefficient vector as $\alpha = (a_3,a_4,a_5,a_6,a_{C_0},a_{C_1},a_{C_2}) \in \mathbb{F}_2^7$. Then the zero-violation condition $V\alpha = 0$, with arithmetic over $\mathbb{F}_2$, is equivalent to
\[
a_4+a_5=0,
\qquad
a_6+a_{C_2}=0.
\]
Hence, $a_4=a_5=t$, $a_6=a_{C_2}=u$, and $a_3,a_{C_0},a_{C_1}$ are free. Applying the mask matrix gives
\[
M\alpha
=
t(m_4+m_5)+u(m_6+m_{C_2})+a_3m_3+a_{C_0}m_{C_0}+a_{C_1}m_{C_1}.
\]
The parameter $u$ changes the generator presentation but not the branch mask, because $m_6=m_{C_2}$. Also $m_3=m_{C_0}$ and $m_4+m_5=m_{C_1}$. Therefore,
\[
\mathcal{K}_F
=
M(\ker V)
=
\left\{
\begin{bmatrix}0\\0\\0\\0\end{bmatrix},
\begin{bmatrix}1\\1\\1\\1\end{bmatrix},
\begin{bmatrix}0\\1\\1\\1\end{bmatrix},
\begin{bmatrix}1\\0\\0\\0\end{bmatrix}
\right\}.
\]
The ranks are
\[
\operatorname{rank}_{\mathbb{F}_2}\begin{bmatrix} M \\ V \end{bmatrix}=4,
\qquad
\operatorname{rank}_{\mathbb{F}_2}(V)=2.
\]
Since $f=1$, the solution count is
\[
\lvert \Xi \rvert = 2^{1+4-2} = 8.
\]
The two independent feasible constrained branch shifts are the mask $\{v_3,v_4,v_5,v_6\}$, obtained from $g_3$ or $g_{C_0}$, and the mask $\{v_4, v_5, v_6\}$, obtained from $g_4\oplus g_5$, or equivalently from the base generator $g_{C_1}$. These generate four feasible constrained codes from a reference solution. The independent free bit at $v_7$ doubles this count, giving eight full feasible branch codes.

\subsection{Concrete Geometric Realization}
\label{sec:example-geometry}

To ground this algebraic structure, we provide a concrete geometric realization in $\mathbb{R}^2$ with exact coordinates. This numerical setup illustrates how the combinatorial conditions established by our theory manifest physically as distance-preserving operations.

\paragraph{Reference Coordinate Embedding.}
We place the initial clique $V_0 = \{v_1, v_2\}$ and the branch decision vertex $v_3$ at
\[
v_1 = (0,0), \quad v_2 = (-2,0), \quad v_3 = (0,2).
\]
The remaining vertices are placed at their reference coordinates:
\[
v_4 = (1, 1), \quad v_5 = (2, 2), \quad v_6 = (2, 0), \quad v_7 = (-1,1).
\]
These coordinates determine the physical edge weights. All discretization distances are positive and can be verified directly: $d_{13} = 2$, $d_{23} = 2\sqrt{2}$, $d_{14} = d_{34} = d_{17} = d_{27} = \sqrt{2}$, $d_{15} = 2\sqrt{2}$, $d_{35} = 2$, $d_{16} = 2$, and $d_{46} = \sqrt{2}$. The two active pruning distances to be tested by our constraint matrices are
\[
d_{36} = \|v_3 - v_6\|_2 = 2\sqrt{2},
\qquad
d_{45} = \|v_4 - v_5\|_2 = \sqrt{2}.
\]
The free vertex $v_7$ contributes only its discretization distances $d_{17}$ and $d_{27}$, which do not appear in the active pruning set $F$ since $v_7 \notin L$.

\paragraph{Physical Reflection and Violation Mechanics.}
We can now trace how individual generator moves deform the embedding and interact with the pruning constraints. First, consider the cone generator $g_3$. The mirror for $v_3$ is the $x$-axis, $\operatorname{aff}(v_1, v_2)$. Activating $g_3$ reflects the downstream cone $\{v_3, v_4, v_5, v_6\}$ across this line. Because all affected vertices are reflected together and the fixed endpoints $v_1,v_2$ lie exactly on the mirror, all discretization and pruning distances are trivially preserved. Similarly, the mirror for $v_7$ is also the $x$-axis. Reflecting $v_7$ to $v_7' = (-1,-1)$ preserves its discretization distances $d_{17}$ and $d_{27}$ and has no contact with the pruning edges. This shows geometrically why the branch decision $v_7$ is completely free, contributing the independent factor $2^f = 2$ to the final solution count.

In contrast, isolated branch shifts that do not belong to the kernel $\ker V$ physically violate the pruning constraints. For instance, the mirror for $v_5$ is the $y$-axis, $\operatorname{aff}(v_1, v_3)$. Mirroring only $v_5$ across the $y$-axis yields $v_5' = (-2, 2)$. While this preserves all discretization distances, it stretches the pruning edge $\{v_4, v_5\}$ to $\|v_4 - v_5'\|_2 = \sqrt{10} \ne \sqrt{2}$, confirming a physical violation. This matches the algebraic violation matrix where $g_5$ has a $1$ in row $(\{v_4,v_5\}, C_1)$. Similarly, the cone generator $g_4$ reflects $v_4$ and its descendant $v_6$ across the $y$-axis, yielding $v_4' = (-1,1)$ and $v_6' = (-2,0)$. This preserves the discretization distance $\{v_4, v_6\}$ and the pruning distance $\{v_3, v_6\}$ (since $v_3$ lies on the $y$-axis mirror), but violates the active edge $\{v_4, v_5\}$ to length $\sqrt{10}$. Lastly, the mirror for $v_6$ is the line $y=x$, $\operatorname{aff}(v_1, v_4)$. Reflecting only $v_6$ gives $v_6' = (0,2)$, which collapses the pruning distance to $v_3$ from $2\sqrt{2}$ to $0$, mapping to the algebraic violation of $g_6$ against $\{v_3,v_6\}$.

\paragraph{Kernel-Induced Cancellation and Dynamic Mirrors.}
The core power of our framework is demonstrated when we apply the nontrivial feasible branch shift $h = \{v_4, v_5, v_6\} \in \mathcal{K}_F$, represented algebraically by the generator combination $g_4 \oplus g_5$ (or equivalently by $g_{C_1}$), where adding the redundant generator pair $g_{C_2} \oplus g_6$ yields the same branch mask. Geometrically, this shift instructs us to reflect $v_4, v_5$, and $v_6$ simultaneously across the $y$-axis ($x=0$), yielding:
\[
v_4' = (-1, 1), \quad v_5' = (-2, 2), \quad v_6' = (-2, 0).
\]
Because all three vertices are reflected by the same isometric reflection $R_{C_1}$ across the $y$-axis, the internal active distance $\|v_4' - v_5'\|_2 = \sqrt{2} = d_{45}$ is preserved. 

The interaction with the dynamic mirror of $v_6'$ reveals the beauty of this cancellation. In the original embedding, $v_6$ is placed using the mirror $\operatorname{aff}(v_1, v_4)$. Under the branch shift, $v_4$ moves to $v_4'$, so the dynamic mirror for the subsequent step rotates to $\operatorname{aff}(v_1, v_4')$, which is the line $y = -x$. The coordinates $v_6' = (-2,0)$ are the exact reflection of the reference $v_6 = (2,0)$ across the $y$-axis. Because $v_6'$ is placed symmetrically with respect to the rotated mirror $y = -x$, the discretization distances are preserved. Crucially, the pruning distance to the fixed vertex $v_3 = (0,2)$ is preserved:
\[
\|v_3 - v_6'\|_2 = \|(0,2) - (-2,0)\|_2 = 2\sqrt{2} = d_{36}.
\]
This coordinate realization illustrates how the algebraic kernel $\ker V$ maps perfectly to viable physical realizations in this system, factoring out the representation redundancies and matching the rank-count prediction.

\section{Conclusion}
\label{sec:conclusion}

By reformulating the combinatorial constraints of the Discretizable Distance Geometry Problem (DDGP) in the language of linear algebra over $\mathbb{F}_2$, this work establishes a bridge between local geometric reflections and global counting complexity. The core insight is that the complexity of solution counting under pruning is fundamentally governed by the algebraic interaction of labeled violations. When physical violations are systematically indexed by their corresponding mirror cliques, the space of feasible branch shifts emerges naturally as the projection $\mathcal{K}_F = M(\ker V)$. Whenever a viable reference solution exists and mirror-separation holds, this linear structure determines the total solution count by assigning an independent factor of two to each branch decision that escapes the algebraic reach of the pruning constraints.

This perspective provides a unified explanation for the tractability of classical DMDGP and unpruned DDGP instances, where graph-induced symmetries naturally guarantee full branch viability. In the presence of arbitrary pruning edges, the labeled violation matrix $V$ acts as a precise algebraic filter, identifying exactly which combinations of reflections preserve distance constraints. Crucially, our resulting rank formula cleanly decouples the purely combinatorial structure of the discretization graph from the measure-zero geometric coincidences that characterize degenerate physical embeddings.

Our algebraic framework opens a promising avenue for the design of highly efficient distance geometry algorithms. By shifting the focus from combinatorial search over exponential branch trees to linear operations over $\mathbb{F}_2$, this formulation provides the mathematical foundation for solver paradigms that may reduce branch enumeration. In particular, the algebraic filters developed here can be directly integrated into solvers to detect infeasible branches at early stages, bypassing expensive coordinate computations.

Specifically, future work will focus on exploiting the structured sparsity of the mask matrix $M$ and the labeled violation matrix $V$ to construct and solve the associated $\mathbb{F}_2$ systems efficiently for large-scale discretization networks. Such advancements hold the potential to significantly accelerate distance geometry solvers, demonstrating the power of finite-field linear algebra to drastically reduce the computational complexity of spatial reconstruction problems.

\bibliographystyle{plain}
\bibliography{references}

\end{document}